\def\draft{n}
\documentclass[11pt]{amsart}
\usepackage[headings]{fullpage}
\usepackage{amssymb,epic,eepic,epsfig,amsbsy,amsmath,amscd,color}
\usepackage[all]{xy}
\usepackage{graphicx}
\usepackage{texdraw}
\usepackage{url}
\usepackage{bbm}
\usepackage{mathrsfs}
\newcommand{\sbullet}{%
  \hbox{\fontfamily{lmr}\fontsize{.4\dimexpr(\f@size pt)}{0}\selectfont\textbullet}}

\makeatother

\def\printname#1{
        \if\draft y
                \smash{\makebox[0pt]{\hspace{-0.5in}
                        \raisebox{8pt}{\tt\tiny #1}}}
        \fi
}

\def\printname#1{
        \if\draft y
                \smash{\makebox[0pt]{\hspace{-0.5in}
                        \raisebox{8pt}{\tt\tiny #1}}}
        \fi
}

\newlength{\standardunitlength}
\catcode`\@=11
\long\def\@makecaption#1#2{%
     \vskip 10pt

\setbox\@tempboxa\hbox{
       \small\sf{\bfcaptionfont #1. }\ignorespaces #2}%
     \ifdim \wd\@tempboxa >\captionwidth {%
         \rightskip=\@captionmargin\leftskip=\@captionmargin
         \unhbox\@tempboxa\par}%
       \else
         \hbox to\hsize{\hfil\box\@tempboxa\hfil}%
     \fi}
\font\bfcaptionfont=cmssbx10 scaled \magstephalf
\newdimen\@captionmargin\@captionmargin=2\parindent
\newdimen\captionwidth\captionwidth=\hsize
\catcode`\@=12

\def\lbl#1{\label{#1}\printname{#1}}

                        \theoremstyle{plain}

\newtheorem{theorem}{Theorem}[section]

\newtheorem{lemma}[theorem]{Lemma}
\newtheorem{corollary}[theorem]{Corollary}
\newtheorem{proposition}[theorem]{Proposition}
\newtheorem{conjecture}{Conjecture}

\newtheorem{definition}{Definition}

\theoremstyle{definition}

\newtheorem{remark}{Remark}[section]

\def\BC{\mathbb C}

\def\BZ{\mathbb Z}
\def\BR{\mathbb R}

\def\CR{\mathcal R}

\def\al{\alpha}

\def\be { \begin{equation} }
\def\ee { \end{equation} }

\newcommand\no[1]{}

\def\bP{{\mathbf P}}

                  \def\uu{\theta}
      \def\nc{\newcommand}

                  \nc\FI[2]{\begin{figure}
    \begin{center}\input{#1.pstex_t}\end{center}
    \caption{#2}
    \lbl{#1}
  \end{figure}}
\nc\FIG[3]{\begin{figure}
    \includegraphics[#3]{#1.eps}
    \caption{#2}
    \lbl{fig:#1}
    \end{figure}}
\nc\FF[3]{\begin{figure}
    \includegraphics[#3]{#1.eps}
    \caption{#2}
    \lbl{#1}
    \end{figure}}

    \nc\FIGc[3]{\begin{figure}[htpb]
    \includegraphics[height=#3]{#1.eps}
    \caption{#2}
    \lbl{fig:#1}
    \end{figure}}

    \nc\FIGh[3]{\begin{figure}[htpb]
    \includegraphics[height=#3]{draws/#1.eps}
    \caption{#2}
    \lbl{fig:#1}
    \end{figure}}

\def\cS{\mathscr S}

\def\cF{\mathcal F}

\def\cP{\mathcal P}

\def\Id{\mathrm{Id}}

\def\bQ{{\mathbf Q}}

\def\SM{(\Sigma,\cP)}
\def\pS{\partial \Sigma}

\def\cI{\mathcal I}

\def\cD{\mathcal D}

                  \def\Zq{{\BZ[q^{\pm 1}]}}

\begin{document}

\title[Positivity of skein algebras]{On positivity of Kauffman bracket skein algebras of surfaces }

\author[Thang  T. Q. L\^e]{Thang  T. Q. L\^e}
\address{School of Mathematics, 686 Cherry Street,
 Georgia Tech, Atlanta, GA 30332, USA}
\email{letu@math.gatech.edu}

\thanks{Supported in part by National Science Foundation. \\
2010 {\em Mathematics Classification:} Primary 57N10. Secondary 57M25.\\
{\em Key words and phrases: Kauffman bracket skein module, positive basis.}}

\begin{abstract}
We show that the Chebyshev polynomials form a basic block of any positive basis of the Kauffman bracket skein algebras of surfaces.

\end{abstract}

\maketitle

\def\cI{\mathcal I}
\def\Id{\mathrm{Id}}
\def\BA{\mathbb A}
\def\cB{\mathcal B}
\def\sM{\mathcal M}
\section{Introduction}

\subsection{Positivity of Kauffman bracket skein algebras}Let $\Zq$ be the ring of Laurent polynomial in an indeterminate $q$ with integer coefficients.
Suppose $\Sigma$ is an oriented surface.
The Kauffman bracket skein algebra $\cS_\Zq(\Sigma)$ is a $\Zq$-algebra generated by unoriented framed links in $\Sigma \times [0,1]$ subject the Kauffman skein relations \cite{Kauffman} which are recalled in Section \ref{sec:prelim}. The skein algebras were introduced by Przyticki \cite{Przy} and independently Turaev  \cite{Turaev:skein} in an attempt to generalize the Jones polynomial to links in general 3-manifolds, and  have connections and applications to many interesting objects like character varieties, topological quantum field theories, quantum invariants, quantum Teichm\"uller spaces, and many others, see e.g. \cite{BW,Bullock,PS,Le:AJ,Muller,Przy,Thurston}.

The quotient $\cS_\BZ(\Sigma)= \cS_\Zq(\Sigma)/(q-1)$, which is a $\BZ$-algebra, is called the skein algebra at $q=1$.

For $R=\BZ$ let $R_+ = \BZ_+$, the set of non-negative integers, and for $R=\Zq$ let $R_+= \BZ_+[q^{\pm1}]$, the set of Laurent polynomials with non-negative coefficients, i.e. the set of $\BZ_+$-linear combinations of integer powers of $q$.
An $R$-algebra, where $R=\Zq$ or $R=\BZ$, is said to be {\em positive} if it is free as an $R$-module and has a basis $\cB           
$
in which the structure constants are in $R_+$, i.e. for any $b,b', b''\in \cB$, the $b$-coefficient of the product $b'b''$ is in $R_+$. Such a basis $\cB$ is called a {\em positive} basis of the algebra.

The important positivity conjecture of Fock and Goncharov \cite{FockGoncharov} states that $\cS_R(\Sigma)$ is positive for the case $R=\Zq$ and $R=\BZ$.
The first case obviously implies the second.
The second case, on the positivity over $\BZ$, was proved by D. Thurston \cite{Thurston}. Moreover, D. Thurston make precise the positivity conjecture by specifying the positive basis as follows.

A {\em normalized sequence of polynomials over $R$} is a sequence $\bP=(P_n(t)_{n\ge 0})$, such that for each $n$, $P_n(t)\in R[t]$ is a monic polynomial of degree $n$.
Every normalized sequence $\bP$ over $R$ gives rise in a canonical way a basis $\cB_\bP$ of the $R$-module $\cS_R(\Sigma)$, see Section \ref{sec:prelim}.

\begin{definition}
A sequence $\bP=(P_n(t)_{n\ge 0})$ of polynomials in $R[t]$ is {\em positive over $R$} if it is normalized and the basis $\cB_\bP$ is a positive basis of  $\cS_R(\Sigma)$, for any oriented surface $\Sigma$.
\end{definition}

 Then D. Thurston proved the following.
  \begin{theorem}[\cite{Thurston}] \lbl{r.Thurston}
The sequence $(T_n)$ of Chebyshev's polynomials is positive over $R=\BZ$.
\end{theorem}
Here in this paper, Chebyshev's polynomials  $T_n(t)$ are defined recursively by
$$ T_0(t) =1, \ T_1(t)=t, T_2(t)= t^2-2, \quad T_n(t) = t T_{n-1}(t) - T_{n-2}(t) \quad \text{for \ } \ n \ge 3.$$
If $T_0$ is the constant polynomial 2, then the above polynomials are Chebyshev's polynomials of type 1. Here we have to set $T_0=1$ since by default, all normalized sequence begins with 1.

D. Thurston suggested the following conjecture, making precise the positivity conjecture.

\begin{conjecture}
\lbl{conj.1}
The sequence $(T_n)$ of Chebyshev's polynomials is positive over $\Zq$.
\end{conjecture}

We will show
\begin{theorem} \lbl{r.main1}
Let  $R=\BZ$ or $R=\Zq$ and  $\bP= (P_n(t)_{n\ge 0})$ be  a sequence of polynomials in $R[t]$.
If\  $\bP$ is positive over $R$, then $P_n(t)$ is an $R_+$-linear combination of\  $T_0(t), T_1(t), \dots, T_n(t)$. Besides, $P_1(t)=T_1(t)=t$.
\end{theorem}

For the case $R=\BZ$, this result complements well Theorem \ref{r.Thurston} above, as they together claim that the sequence of Chebyshev polynomials is the minimal one in the set of positive sequence over $\BZ$.
For $R=\Zq$, our result says that the sequence of Chebyshev polynomial should be the minimal positive sequence.   A stronger version of Theorem \ref{r.main1} is proved in Section \ref{sec:2}.

\begin{remark}The positivity conjecture  considered here is  different from the one discussed in \cite{MSW}, which claims that every element of a certain basis is a $\BZ_+$-linear combination of monomials of cluster variables.
\end{remark}

\subsection{Marked surfaces} {\em A marked surface} is a  pair $\SM$, where $\Sigma$ is a compact oriented surface with (possibly empty) boundary $\pS$ and $\cP$ is a finite set in $\pS$.
G. Muller \cite{Muller} defined the skein algebra $\cS_R\SM$, extending the definition from surfaces to marked surfaces.
When $R=\BZ$, this algebra had been known earlier, and actually, the above mentioned result of D. Thurston (Theorem~\ref{r.Thurston}) was proved also for the case of marked surfaces.  However, there are two types of basis generators of the skein algebras, namely loops and arcs, and one needs two normalized sequences of polynomials $\bP$ and $\bQ$ to define an $R$-basis of the skein algebra $\cS_R\SM$. Here $\bP$ is applicable to  loops, and $\bQ$ is applicable to arcs,  see Section \ref{sec:marked}.
A pair of sequences of polynomials $(\bP,\bQ)$ are {\em positive over $R$} if they are normalized and the basis they generate is positive for any marked surface. D. Thurston result says that with $\bP= (T_n)$, the sequence of Chebyshev polynomials, and $\bQ=(Q_n)$
defined by $Q_n(t)=t^n$, the pair $(\bP,\bQ)$ are positive over $\BZ$. We obtained also an extension of Theorem \ref{r.main1} to the case of marked surface as follows.

\begin{theorem} \lbl{r.main2}
Let  $R=\BZ$ or $R=\Zq$. Suppose a pair $(\bP,\bQ)$ of sequences of  polynomials in $R[t]$, $\bP= (P_n(t)_{n\ge 0})$ and $\bQ= (Q_n(t)_{n\ge 0})$, are positive over $R$.
Then $P_n(t)$ is an $R_+$-linear combination of \ $T_0(t), T_1(t), \dots, T_n(t)$ and $Q_n(t)$ is an $R_+$-linear combination of $1, t,  \dots, t^n$. Moreover, $P_1(t)=Q_1(t)=t$.
\end{theorem}

\subsection{Acknowledgment} The author would like to thank D. Thurston for valuable discussions. The author also thanks the anonymous referees for useful comments which in particular lead to the current, stronger  formulation of Theorem \ref{thm.1f}.
This work is supported in part by the NSF.
\subsection{Plan of the paper} In Section \ref{sec:prelim} we recall basic facts about the Kauffman bracket skein algebras of surfaces. We present the proofs of Theorem~\ref{r.main1} and Theorem~\ref{r.main2} in respectively Section~\ref{sec:2} and Section~\ref{sec:marked}.

\section{Skein algebras} \label{sec:prelim} We recall here basic notions concerning the Kauffman bracket skein algebra of a surface.

\subsection{Ground ring} Throughout the paper we work with a ground ring $R$ which is more general than $\BZ$ and $\Zq$. We will assume that 
 $R$ is a commutative domain over $\BZ$ containing an invertible element $q$ and a subset $R_+$ such that
  \begin{itemize}
  \item $R_+$ is closed under addition and multiplication
  \item $q, q^{-1} \in R_+$, and
  \item $R_+ \cap (- R_+)= \{0\}$.
  \end{itemize}

  For example, we can take $R= \BZ$ with $q=1$ and $R_+= \BZ_+$, or $R=\Zq$ with $R_+= \BZ_+[q^{\pm 1}]$. The reader might think of $R$ as one of these two rings.

\subsection
{Kauffman bracket skein algebra} Suppose $\Sigma$ is an oriented surface.
The {\em Kauffman bracket skein module}  $\cS_R(\Sigma)$ is  the $R$-module freely spanned by isotopy classes of non-oriented framed links in $\Sigma \times [0,1]$ modulo the {\em skein relation} and  the {\em trivial loop relation} described in
Figure \ref{fig:skein}.
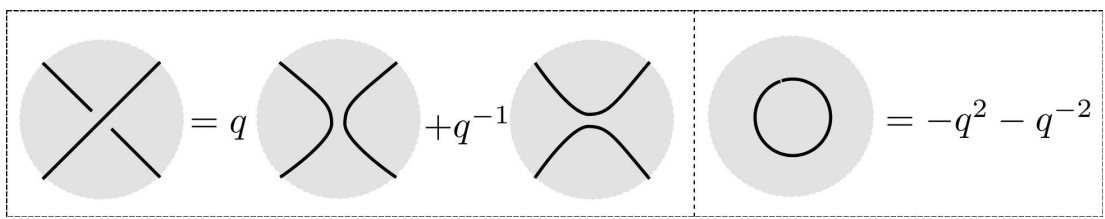
 \FIGc{skein}{Skein relation (left) and trivial loop relation}{2.8cm}
In all Figures, framed links are drawn with blackboard framing. More precisely, the  trivial loop relation says
 if  $L$ is a loop  bounding a disk in $\Sigma\times[0,1]$ with framing perpendicular to the disk, then
 $L= -q^2 -q^{-2}.$ And
the skein relation says $$
 L  = q L_+ + q^{-1} L_-
 $$ if
   $L, L_+, L_-$ are  identical except in a ball in which they look like
     in Figure \ref{fig:skein1}.
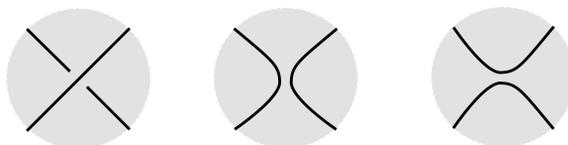
     \FIGc{skein1}{From left to right:  $L, L_+, L_-$.}{1.9cm}

For future reference,  we say that the diagram $L_+$ (resp. $L_-$) in Figure \ref{fig:skein1} is obtained from $L$ by the positive (resp. negative) resolution of the crossing.

\def\cF{\mathcal F}
\def\An{{\mathbb  A}}
\newcommand{\vect}[1]{\overrightarrow{#1}}
\def\cP{\mathcal P}


The $R$-module $\cS_R(\Sigma)$
has an algebra structure where the product of two links $\al_1$ and $\al_2$ is  the result of stacking $\al_1$  atop  $\al_2$ using the cylinder structure of $\Sigma \times [0,1]$. Over $R=\BZ$, the skein algebra $\cS_R(\Sigma)$ is commutative and is closely related to the $SL_2(\BC)$-character variety of $\Sigma$, see \cite{Bullock,PS,Turaev:skein,BFK}. Over $R=\Zq$, $\cS_R(\Sigma)$ is not commutative in general and is closely related to quantum Teichm\"uller spaces \cite{CF}.

\subsection{Bases} We will consider $\Sigma$ as a subset of $\Sigma\times [0,1]$ by identifying $\Sigma$ with $\Sigma \times \{1/2\}$.  As usual, links in $\Sigma$, which are  closed 1-dimensional non-oriented submanifolds of $\Sigma$,  are considered up to ambient isotopies of $\Sigma$. A link $\al$ in $\Sigma$ is  {\em an essential } if it has no trivial component, i.e. a component bounding a disk in $\Sigma$. By convention, the empty set is considered an essential link. The framing of a link is
{\em vertical} if at every point $x$, the framing is parallel to $x \times [0,1]$, with the direction equal to the positive direction of $[0,1]$.

By \cite[Theorem 5.2]{PS},  $\cS_R(\Sigma)$ is  free over $R$ with  basis the set of all isotopy classes of essential links in $\Sigma$ with vertical framing.
This basis can be parameterized as follows.  An {\em integer lamination} of $\Sigma$ is an unordered collection $\mu=(n_i,C_i)_{i=1}^m$, where

\begin{itemize}
\item each $n_i$ is a positive integer
\item each $C_i$ is a non-trivial  knot in $\Sigma$
\item no two $C_i$ intersect
\item no two $C_i$ are ambient isotopic.
\end{itemize}

For each integer lamination $\mu$, define an element $b_\mu\in \cS_R(\Sigma)$ by
$$ b_\mu= \prod_{i} (C_i)^{n_i}.$$
Then the set of all $b_\mu$, where $\mu$ runs the set of all integer laminations including the empty one, is the above mentioned basis of $\cS_R(\Sigma)$.

Suppose  $\bP=(P_n(t)_{n\ge 0})$ is a normalized sequence of polynomials in  $R[t]$. Then we can twist the basis element $b_\mu$ by $\bP$ as follows. Let
$$ b_{\mu,\bP} := \prod_{i} P_{n_i}(C_i).$$
As $\{ P_n(t)\} \mid n \in \BZ_+ \}$, just like $\{ z^n\}$, is a basis of $R[t]$, the set $\cB_\bP$ of all $b_{\mu,\bP}$, when $\mu$ runs the set of all integer laminations, is a free $R$-basis of $\cS_R(\Sigma)$.

\section{Theorem \ref{r.main1} and its stronger version}
\lbl{sec:2}

We present here the proof of a stronger version of Theorem \ref{r.main1}, using a result from \cite{Le:central} which we recall first.

\subsection{Skein algebra of the  annulus} Let $\An \subset \BR^2$ be the annulus $\An= \{   x \in \BR^2,  1 \le | x| \le 2\}$.
Let $z$ be the core of the annulus defined by $z = \{ {x}, |x | = 3/2\}$.
It is easy to show that, as an algebra,  $ \cS_R(\An) = R[z]$.

\def\ppp{$(p_1 \!\!
\leftrightarrow
\!\!
 p_2)$}
 \def\ppp{$\Aio$}

\label{sec.2}
\def\Aio{{\mathbb A_{io}}}
\def\An{\mathbb A}

Let $p_1= (0,1)\in \BR^2$ and $p_2=(0,2)\in \BR^2$, which are points in $\partial \An$. Then
$\Aio=(\An, \{p_1, p_2\})$ is an example of a marked surface. See Figure \ref{fig:Aio}, which also depicts  the arcs $\uu_0$, $\uu_{-1}$, $\uu_{2}$. Here $\uu_0$ is the straight segment $p_1p_2$, and $\uu_n$, for $n\in \BZ$,  is an arc properly embedded in $\An$ beginning at $p_1$ and winding clockwise (resp. counterclockwise)  $|n|$ times if $n\ge 0$ (resp. $n < 0$) before getting to $p_2$.
\FIGc{Aio}
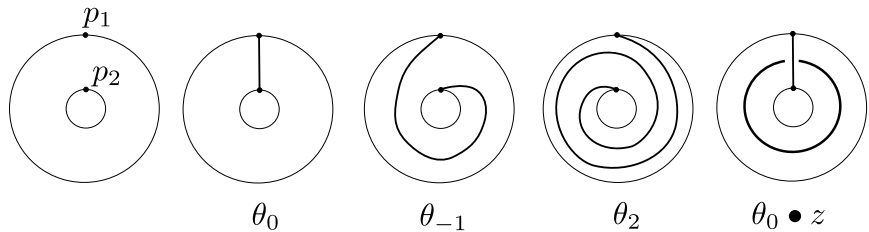
{Marked annulus $\Aio$, arcs $\uu_0$, $\uu_{-1}$, $\uu_{2}$, and element $\uu_0\bullet z$}{3cm}

 {A \em \ppp-arc} is a proper embedding of  the interval $[0,1]$  in $\An \times [0,1]$ equipped with a framing such that one end point is in $p_1 \times [0,1]$ and the other is in $p_2 \times [0,1]$, and  the framing is vertical at both end points.  {\em A \ppp-tangle} is a disjoint union of a \ppp-arc and a (possibly empty) framed link in $\An \times [0,1]$. Isotopy of \ppp-tangles are considered in the class of \ppp-tangles.

Let $\cS'_R(\Aio)$ be the $R$-module spanned by isotopy classes of \ppp-tangles modulo the usual skein relation and the trivial knot relation. As usual, each $\uu_n$ is equipped with the vertical framing, and is considered as an element of $\cS'_R(\Aio)$.

For $\al\in \cS_R(\An)$ let $\uu_0 \bullet \al\in \cS'_R(\Aio)$ be the element obtained by placing $\uu_0$ on top of $\al$, see Figure~\ref{fig:Aio} for an example. In \cite{Le:central} we proved the following.
\begin{proposition}
\lbl{r.trans}
For any integer $n \ge 1$, we have
\be \lbl{eq.1}
\uu_0 \bullet T_n(z) = q^n \uu_n + q^{-n} \uu_{-n}.
\ee
\end{proposition}


\def\Tp{{\mathbb T}_{\mathrm{punc}}}

\def\BD{\mathbb D}
\def\All{\An_{oo}}
\def\cD{\mathcal D}

\def\cH{\mathcal H}
\def\Smm{\Sigma_{1,1}}

\subsection{Stronger version  of Theorem \ref{thm.1f}}

We say that a sequence $\bP=(P_n)$ of normalized polynomials in $R[t]$ is {\em positive for $\Sigma$} over $R$ if  $\cB_\bP$ is a positive basis for $\cS_R(\Sigma)$. Thus, $\bP=(P_n)$ is positive if and only if it is positive for any oriented surface.

Theorem \ref{r.main1} follows from the following stronger statement.

 \begin{theorem}  \lbl{thm.1f}
 Suppose  $\bP=(P_n(t))$ is positive for a non-planar oriented surface $\Sigma$. Then for every $n \ge 0$, $P_n(t)$ is an $R_+$-linear combination of $T_k(t)$ with $k \le n$.
 Besides, $P_1(t)=t$.
 \end{theorem}

 \begin{proof}  Since $\Sigma$ is non-planar, there are two simple closed curves $z,z'\subset \Sigma$ which intersect transversally at one point. We identify a small tubular neighborhood of $z$ with the annulus $\An$ such that $z'\cap \An$ is the segment $p_1 p_2$, see Figure \ref{fig:pTor}.  Note that $z$ and $z'$, as homology classes in $H_1(\Sigma,\BZ)$,  are linearly independent over $\BZ$.
 
 \no{
 We  present $\Smm$ as the union
$\Smm = \An \cup \An'$, where  $\An'$ is an annulus with core $z'$, as in Figure~\ref{fig:pTor}, with  $z'\cap \An$ equal to the arc $p_1p_2$.   Recall that $z$ is the core of the annulus $\An$.
}

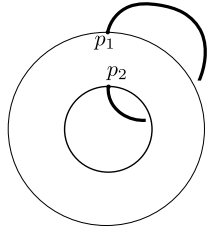
\FIGc{pTor}{The union $\An \cup z'$. The bold arc is  $z'\setminus \An$.}{3cm}

 By definition 
 $P_1(t)=t+a$, $a \in R$.
 Fix an integer $n\ge 1$. There are $c_k \in R$ such that
 $$ P_n(t) = \sum_{k=0}^n c_k T_k(t).$$
 Let  $z_{1,k}$ be the curve $(z' \setminus \An)\cup \uu_k$. We have
\begin{align} \lbl{eq.6}
  P_1(z') P_n(z) & = a P_n(z) + z'\sum_{k=0}^n c_k  T_k(z) = a P_n(z) +  c_0 z' + \sum_{k=1}^n c_k z' T_k(z) \nonumber \\
  & = a P_n(z) +  c_0 z' + \sum_{k=1}^n c_k   \left( q^k z_{1,k}  +   c_kq^{-k} z_{1,-k}\right) \end{align}
 where the last equality follows from  Proposition \ref{r.trans}. Using homology, one  sees that in the collection
$\{ z, z', z_{1,k}, z_{1,-k} \mid  k \ge 1\}$  every curve is a  non-trivial knot in $\Sigma$, and any two of them are non-isotopic.
 Rewriting $t = P_1(t) -a$, we have
 \begin{align*}
 P_1(z') P_n(z) & = a P_n(z) +  c_0  (P_1(z') - a) + \sum_{k=1}^n c_k \left [ q^k (P_1( z_{1,k}) -a) + q^{-k} (P_1(z_{1,-k})-a)  \right]\\
 &  = a P_n(z) +  c_0  \, P_1(z')   + \sum_{k=1}^n c_k \left [ q^k (P_1( z_{1,k}) + q^{-k} (P_1(z_{1,-k})  \right]  - a c_0  -\sum_{k=1}^n  a c_k (q^k + q^{-k}).
 \end{align*}
By the positivity, the coefficients of $P_1(z')$,  $P_n(z)$, $P_1(z_{1,k})$ (for every $k \ge 1$) and the constant coefficient are in $R_+$. This shows that $a, c_k \in R_+$ for every $k$, and
\be
d:=- a c_0  -\sum_{k=1}^n  a c_k (q^k + q^{-k}) \in R_+.
\nonumber
\ee
Note that $-d = a( c_0 + \sum_{k=1}^n c_k (q^k + q^{-k})) \in R_+$, since $a, c_k \in R_+$.
Thus, both  $d$ and $-d$ are in $R_+$. Then $d=0$ by the assumption on $R$.  It follows that
\be
\lbl{eq.neq0}
a(  c_0 + \sum_{k=1}^n c_k (q^k + q^{-k})) =0.
\ee
{\bf Claim.} Suppose $x,y\in R_+$ such that $x+y=0$. Then $x=y=0$.  \\
{\em Proof of Claim.}  We have $x=-y$. Thus, $y$ and $-y=x$ are in $ R_+$, implying $y=0$. Then $x=0$.

The claim means that if $x,y \in R_+$ and $x\neq 0$, then $x+y \neq 0$. In particular, $q^n+q^{-n} \neq 0$.  Note that $c_n=1$ because $P_n(t)$ is a monic polynomial. All the summands in the sum
 $$ c_0 + \sum_{k=1}^n c_k (q^k + q^{-k}) $$
 are in $R_+$, and the summand with $k=n$ is non-zero. Hence the above sum is non-zero. Since $R$ is a domain, from \eqref{eq.neq0} we conclude that $a=0$. This completes the proof of the theorem.
 \end{proof}

 \no{
   Suppose $\Sigma$ is a non-planer oriented connected surface. Then there is an embedding $\Sigma_{1,1} \hookrightarrow \Sigma$ which is injective on homology. 
   Since every non-planar surface contains a  $\Sigma_{1,1}$, we have the following consequence, which was pointed out to the author by the referee.
   
   \begin{corollary}
   Suppose  $\bP=(P_n(t))$ is positive for a non-planar surface. Then for every $n \ge 0$, $P_n(t)$ is an $R_+$-linear combination of $T_k(t)$ with $k \le n$.
 Besides, $P_1(t)=t$. \end{corollary}
 }

\newcommand{\qbinom}[2]{\left[\begin{matrix}
#1 \\ #2
\end{matrix}  \right]}

\def\sq{Q}
\def\pT{\text{p-$\mathbb T$}}
\def\En{{ _nE}}

\section{Marked surfaces}
\lbl{sec:marked}

\subsection{Skein algebra of marked surfaces} Suppose $\SM$ is a marked surface, i.e. $\cP\subset \pS$ is a finite set.
{\em A   framed 3D $\cP$-tangle $\al$ in $\Sigma\times [0,1]$} is a framed proper embedding of a 1-dimensional non-oriented compact manifold into $\Sigma\times [0,1]$ such that $\partial \al \subset \cP \times [0,1]$, and the framing at every boundary point of $\al$ is vertical.
 Two  framed 3D $\cP$-tangles are {\em isotopic} if they are
 isotopic through the class of framed 3D $\cP$-tangles.

 Just like the link case, a framed 3D $\cP$-tangle $\al$  is depicted by its diagram on $\Sigma$, with vertical framing on the diagram. Here a diagram of $\al$ is a projection of $\al$ onto $\Sigma$ in general position, with an order of strands at every crossing. Crossings come in two types. If the crossing is in
  $\Sigma \setminus \cP$, then it is a usual double point (with usual over/under information). If the crossing is a point in $\cP$,  there may be two or more number of strands, which are ordered. The order indicates which strand is above which. When there are two strands, the lower one is depicted by a broken line, see e.g. Figure \ref{fig:arc}.

 Let $\cS_R\SM$ be the $R$-module spanned by the set of isotopy classes of framed 3D $\cP$-tangles in $\Sigma\times [0,1]$ modulo the skein relation, the trivial knot relation (Figure~\ref{fig:skein}), and the trivial arc relation of Figure~\ref{fig:arc}.
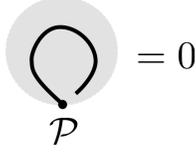
 \FIGc{arc}{Trivial arc relation: If  a framed 3D $\cP$-tangle $\al$ has a trivial arc, then $\al=0$.
 }{2cm}

The following relation holds, see  \cite{Le:qtrace}.
\begin{proposition} \lbl{r.boundary}
In $\cS\SM$, the reordering relation depicted in Figure \ref{fig:boundary} holds.
\end{proposition}
\FIGc{boundary}{ Reordering relation.}{1.4cm}

 Again one defines the product $\al_1 \al_2$ of of two skein elements $\al_1$ and $\al_2$ by stacking $\al_1$ above $\al_2$. This makes $\cS_R\SM$ an  $R$-algebra, which was first defined by Muller \cite{Muller}.

\subsection{Basis for $\cS_R\SM$} {\em A $\cP$-arc}  is an immersion $\al: [0,1]\to \Sigma$ such that $\al(0), \al(1)\in \cP$ and
 the restriction of $\al$ onto $(0,1)$ is an embedding into  $\Sigma \setminus \cP$. {\em A $\cP$-knot}  is an embedding of $S^1$ into  $\Sigma \setminus \cP$. A $\cP$-knot or a $\cP$-arc is {\em trivial} if it bounds a disk in $\Sigma$. Two $\cP$-arcs (resp. $\cP$-knots) are
 {\em $\cP$-isotopic} if they are isotopic in the class of $\cP$-arcs  (resp. $\cP$-knots).

A $\cP$-arc  is called {\em boundary} if it can be isotoped into $\pS$, otherwise it is called {\em inner}.
We consider every $\cP$-arc and every $\cP$-knot as an element of $\cS_R\SM$ by equipping it with the vertical framing. In the case when the two ends of a $\cP$-arc are the same point $p\in\cP$, we order the left strand to be above the right one.

An {\em integer $\cP$-lamination} of $\Sigma$ is an unordered collection $\mu=(n_i,C_i)_{i=1}^m$, where
\begin{itemize}
\item each $n_i$ is a positive integer
\item each $C_i$, called a component of $\mu$, is a non-trivial  $\cP$-knot or a non-trivial $\cP$-arc
\item no two $C_i$ intersect in  $\Sigma\setminus \cP$
\item no two $C_i$ are  $\cP$-isotopic.
\end{itemize}

In an integer $\cP$-lamination $\mu=(n_i,C_i)_{i=1}^m$, a knot component commutes with any other component
(in the algebra $\cS_R\SM$), while a $\cP$-arc component $C$  $q$-commutes with any other component $C'$ in the sense that $CC'= q^k C'C$ for a certain $k\in \BZ$. The $q$-commutativity follows from Proposition \ref{r.boundary}.
Hence the element

\be
\lbl{eq.basis}
 b_\mu= \prod_{i} (C_i)^{n_i}
 \ee
 is defined up to a factor $q^k, k\in \BZ$. To make $b_\mu$ really well-defined, we will fix once and for all a total order on the set of all $\cP$-arcs in $\Sigma$. Now define $b_\mu$ by \eqref{eq.basis}, where the product is taken in this order.   To simplify some proofs, we further assume that  in the total order any boundary $\cP$-arc comes before any inner $\cP$-arc, although this is not necessary.
  
It follows from \cite[Lemma 4.1]{Muller} that the set of all $b_\mu$, where $\mu$ runs the set of all integer $\cP$-laminations including the empty one, is the a free $R$-basis of $\cS_R\SM$.

Suppose  $\bP=(P_n)$ and $\bQ=(Q_n)$ are  normalized sequences of polynomials in  $R[t]$. Define
\be
\lbl{eq.basis2}
 b_{\mu,\bP, \bQ}= \prod_{i} F_{n_i}(C_i),
 \ee
where the product is taken in the above mentioned order, and
\begin{itemize}
\item $F_{n_i}=P_{n_i} $ if $C_i$ is a knot,
\item  $F_{n_i}= Q_{n_i}$ if $C_i$ is an inner $\cP$-arc,
\item  $F_{n_i}(C_i) = (C_i)^{n_i}$ if $C_i$ is a boundary $\cP$-arc.
\end{itemize}

Then the set  $\cB_{\bP, \bQ}$ of all  $b_{\mu,\bP, \bQ}$, where $\mu$ runs the set of all integer $\cP$-laminations, is a free $R$-basis of $\cS_R\SM$. A pair of sequences of polynomials $(\bP,\bQ)$ are {\em positive over $R$} if they are normalized and the basis $\cB_{\bP, \bQ}$ is positive for any marked surface.

\def\Df{\mathbb D, \cP}

\subsection{Positive basis in quotients} The following follows right away from the definition.
\no{
\begin{lemma}
Suppose $\cB$ is a positive basis of an $R$-algebra $A$ and $B\subset A$ is an $\CR$-submodule which respects the base $\cB$ in the sense that $B$ is freely  $R$-spanned by $I \cap \cB$. Let $\pi: A \to A/B$ be the natural projection. Then $\pi(\cB\setminus B)$ is a basis of $A/B$, and the coordinates of $\pi(b b')$ for any $b,b'\in \cB$ in this basis are in $R_+$.
\end{lemma}
}
\begin{lemma}
\lbl{r.quotient}
 Let $\cB$ be a positive basis of an $R$-algebra $A$ and $I\subset A$ be an ideal of $A$, with $\pi: A \to A/I$  the natural projection. Assume that $I$
 respects the base $\cB$ in the sense that $I$ is freely  $R$-spanned by $I \cap \cB$.  Then $\pi(\cB\setminus I)$ is a positive basis of $A/I$.
\end{lemma}
\subsection{Ideal generated by boundary arcs}
\begin{lemma}
\lbl{r.ideal}
Suppose $\al_1,\dots,\al_l$ are boundary $\cP$-arcs, and $I=\sum_{i=1}^l \al_i \cS_R\SM$.

 (a) The set $I$ is a 2-sided ideal of $\cS_R\SM$.

 (b) For any normalized sequences $\bP, \bQ$ of polynomials in  $R[t]$, $I$
 respects the basis $\cB_{\bP,\bQ}$.
\end{lemma}
\begin{proof}
(a) Suppose $\al$ is a boundary $\cP$-arc. Since $\al$ $q$-commutes with any basis element $b_\mu$ defined by \eqref{eq.basis}, $\al\cS_R\SM$ is a 2-sided ideal of $\cS_R\SM$. Hence $I$ is also a 2-sided ideal.

(b) In the chosen order,  the boundary $\cP$-arcs come before any inner arcs $\cP$-arc.
Hence $I$ is freely $R$-spanned by all $b_{\mu,\bP,\bQ}$ such that $\mu$ has one of the $\al_i$ as a component.
 Thus, $I$ respects $\cB_{\bP,\bQ}$.
\end{proof}

\subsection{Proof of Theorem \ref{r.main2}}
We will prove the following stronger version of Theorem \ref{r.main2}.

 \begin{theorem}\lbl{r.main2a}
 Suppose a pair $(\bP,\bQ)$ of sequences of  polynomials in $R[t]$, $\bP= (P_n(t)_{n\ge 0})$ and $\bQ= (Q_n(t)_{n\ge 0})$, are positive over $R$.
Then $P_n(t)$ is an $R_+$-linear combination of \ $T_0(t), T_1(t), \dots, T_n(t)$ and $Q_n(t)$ is an $R_+$-linear combination of $1, t,  \dots, t^n$. Moreover, $P_1(t)=Q_1(t)=t$.
\end{theorem}

\begin{lemma} One has $Q_1(t)= t$.
\end{lemma}
\begin{proof}
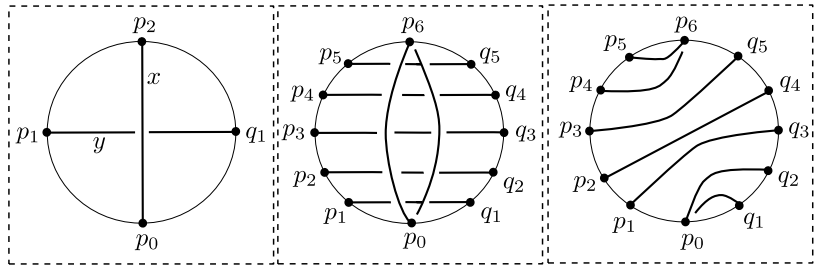
\FIGc{disk1}{Left: $\BD_1$, with product $xy$. Middle: $\BD_5$, with product $x^2 y_5$. Right: Element $z_{2,5}$}{3.5cm}

Consider the marked surface $\BD_1$, which is a disk with 4 marked points $p_0, p_1, p_2$ and $q_1$ as in Figure \ref{fig:disk1}.
Let $x$ be the arc $p_0p_2$ and $y$ the arc
 $ p_1 q_1$. Suppose $Q_1(t)= t + a$, where $a\in R$. Using the skein relation to resolve the only crossing of $xy$, we get
\be
\lbl{eq.11}
 Q_1(x) Q_1(y) = a Q_1(x) + a Q_1(y) - a^2 \mod I_\partial.
\ee
Here $I_\partial$ is the ideal of $\cS_R(\BD_1)$ generated by all the boundary $\cP$-arcs, which respects $\cB_{\bP,\bQ}$ by Lemma~\ref{r.ideal}. Lemma~\ref{r.quotient} and Equ.~\eqref{eq.11} show that
$ a \in R_+$ and  $- a^2 \in R_+$. Hence both $a^2$ and $-a^2$ are $R_+$, which implies $a=0$.
\end{proof}

Now fix an integer number $n \ge 2$.  Let $\BD_n$ be the marked surface, which is a disk with $2n+2$ marked points which in clockwise orders are $p_0,p_1,\dots,p_{n+1}, q_n,\dots,q_1$.  We draw $\BD_n$ in the standard plane so that the straight segment $p_0p_{n+1}$, denoted by $x$, is vertical with $p_0$ being lower than $p_{n+1}$, and each straight segment $p_i q_i$ is horizontal. See   Figure~\ref{fig:disk1} for an example of $\BD_5$.
Let $y_n$ be the union of the $n$ straight segments $p_i q_i$, $i=1,\dots,n$. Considering  $x$ as an element of $\cS_R(\BD_n)$, we will present $x^k$ by  $k$ arcs, each going from $p_0$ monotonously upwards to $p_{n+1}$. Any two of these $k$ arcs do not have intersection in the interior of $\BD_n$, and the left one  is above the right one. See an example in  Figure~\ref{fig:disk1} where the diagram of $x^2 y_5$ is drawn. Then the diagram of $x^k y_n$ has exactly $kn$ double points.
 Let  $z_{k,n}$ be obtained from this diagram of $x^k y_n$ by negatively resolving all the crossings, again see Figure~\ref{fig:disk1}.

 For each $i=0,\dots,n-1$ let $\gamma_i$ be the arc $p_i p_{i+1}$, which is a boundary $\cP$-arc.
By Lemma
\ref{r.ideal}, the set
$$ I := \sum_{i=0}^{n-1} \gamma_i \cS\SM$$
is a 2-sided ideal of $\cS\SM$ respecting $\cB_{\bP,\bQ}$. It is important that the arc $p_n p_{n+1}$ is not in  $I$.
\begin{lemma}\lbl{r.phu2}
For every $k \le n$, one has
\be
\lbl{eq.2d}
 x^k y_n = q^{-kn} z_{k,n} \mod I.
 \ee
\end{lemma}


\def\fu{\mathfrak u}
\begin{proof}  Label the $k$  arcs of the diagram of $x^k$ from left to right by $1, \dots, k$. There are $kn$ crossings in the diagram of $x^k y_n$, and denote by $E_{l,m}$ the intersection of the $l$-th arc of $x^k$ and the arc $p_m q_m$ of $y_n$. Order the set of all $kn$ crossings $E_{lm}$ by the lexicographic pair $(l+m,l)$.

Suppose $\tau$ is one of the $2^{kn}$ ways of resolutions of all the $kn$ crossings. Let $D_\tau$ be the result of the resolution $\tau$, which is a diagram without crossing. Assume $\tau$ has at least one positive resolution. We will prove that $D_\tau\in I$.
 Let $E_{l,m}$ be the smallest crossing at which the resolution is positive.
  In a small neighborhood of $E_{l,m}$, $D_\tau$ has two arcs, with the lower left one denoted by $\delta$, see Figure \ref{fig:positive}.
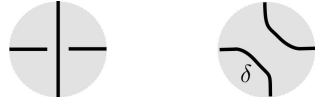
\FIGc{positive}{Left: Crossing $E_{l,m}$. Right: Its positive resolution, and the arc $\delta$.}{1.3cm}

The resolution at any $E_{l', m'} < E_{l,m}$ is negative. These data are enough to determine the arc $\fu$ of $D_\tau$ containing $\delta$, see Figure \ref{fig:proof4}. Namely, if $m \ge l$ then $\fu$ is $\gamma_{m-l}$, and if $m <l$ then $\fu$ is  an arc whose two end points are $p_0$, which is 0.  See Figure \ref{fig:proof5} for an example.
Either case, $D_\tau \in I$.
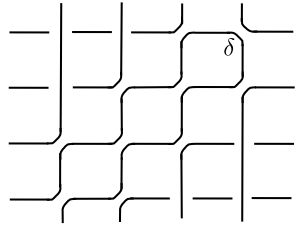
\FIGc{proof4}{The top right corner crossing is $E_{l,m}$, with positive resolution. The arc of $D_\tau$ containing $\delta$ keeps going down the south-west direction.}{3cm}
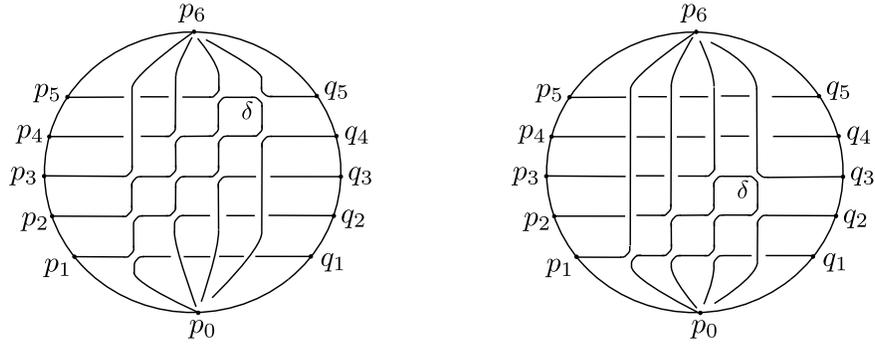
\FIGc{proof5}{Left: The arc $\fu$ with $l=4, m=5$. In this case $\fu=\gamma_1$. 
Right: The arc $\fu$ with $l=4, m=3$. In this case,  both end points of $\fu$ are $p_0$.
}{4.5cm}

Hence, modulo $I$, the only element obtained by resolving all the crossings of $x^k y_n$ is the all-negative resolution one, which is $z_{k,n}$. The corresponding factor coming from the skein relation is $q^{-kn}$. This proves Identity \eqref{eq.2d}.
\end{proof}

\begin{proof}[Proof of Theorem \ref{r.main2a}]
Theorem \ref{thm.1f} implies that $P_n$ is an $R_+$-linear combination of $T_k$ with $k \le n$.

Since $Q_1(t)=t$, each of $z_{k,n}$, $y_n$ in Lemma \ref{r.phu2} is an element of the basis $\cB_{\bP,\bQ}$.

Suppose $Q_n(t)= \sum_{k=0}^n c_k t^k$ with $c_k \in R$. From \eqref{eq.2d}, we have
$$ Q_n(x) y_n =\left( \sum_{k=0}^n c_k x^k \right) y_n = \sum_{k=0}^n c_k q^{-kn} z_{k,n} + I.$$
Note that $z_{k,n}\not \in I$.
Since $I$ respects the basis $\cB_{\bP,\bQ}$, Lemma \ref{r.quotient} shows that
 $c_k \in R_+$ for all $k$. This completes the proof of Theorem \ref{r.main2a}.
\end{proof}

\end{document}